\documentclass[11pt]{article}
\usepackage{amssymb, latexsym, mathtools, amsthm}
\begingroup
    \makeatletter
    \@for\theoremstyle:=definition,remark,plain\do{%
        \expandafter\g@addto@macro\csname th@\theoremstyle\endcsname{%
            \addtolength\thm@preskip\parskip
            }%
        }
\endgroup
\usepackage{enumerate}
\usepackage{subfigure}
\usepackage{pgf,tikz}
\usepackage{pdfsync}
\usepackage{txfonts}
\usepackage{booktabs}
 \usepackage{graphicx}
\definecolor{dnrbl}{rgb}{0,0,0.3}
\definecolor{dnrgr}{rgb}{0,0.3,0}
\definecolor{dnrre}{rgb}{0.5,0,0}
\usepackage[colorlinks=true, citecolor=dnrgr, linkcolor=dnrre, urlcolor=dnrre] {hyperref}
\usepackage{xcolor}
\usepackage{parskip}
\usepackage{tabularx, colortbl}

\usepackage{geometry}
\usepackage{tabulary}
\newcolumntype{K}[1]{>{\centering\arraybackslash}p{#1}}
\newcolumntype{L}[1]{>{\arraybackslash}p{#1}}

 \usepackage[scriptsize, up]{caption}
\usepackage{pgf,tikz}
\usetikzlibrary{decorations, decorations.pathmorphing}
\usetikzlibrary {shapes}

\newcommand{\TTast}{\mathcal{T}^{\ast}} 
\newcommand{\TT}{\mathcal{T}}

\theoremstyle{plain}
\newtheorem{thm}{Theorem}[section]

\newtheorem{prop}[thm]{Proposition}
\newtheorem{lem}[thm]{Lemma}
\newtheorem{coro}[thm]{Corollary}
\newtheorem{defi}[thm]{Definition}
\numberwithin{equation}{subsection}
\makeatletter

\let\c@table\c@figure
\makeatother

\newcommand{\Nat}{\mathbb{N}}

\newcommand{\restr}{\upharpoonright}  
\newcommand{\un}{\uparrow} 
\newcommand{\de}{\downarrow} 
\DeclarePairedDelimiter{\ceil}{\lceil}{\rceil}

\newcommand{\smo}[1]{\mathop{\bf o}\/\left({#1}\right)}

\newcommand{\ml}{Martin-L\"{o}f }
\newcommand{\pz}{$\Pi^0_1$\ }

\newcommand{\eg}{e.g.\ }
\newcommand{\ie}{i.e.\ }
\newcommand{\ce}{c.e.\ }

\newcommand{\pf}{prefix-free }

\renewenvironment{abstract}
 { \normalsize
  \list{}{
    \setlength{\leftmargin}{.0cm}%
    \setlength{\rightmargin}{\leftmargin}%
    }%
  \item {\bf \abstractname.} \relax}
 {\endlist}

\newcommand{\PP}{\mathcal{P}}

\newcommand{\KG}{Ku\v{c}era-G\'{a}cs }

\title{Limits of the \KG coding method\thanks{Barmpalias was supported by the 
1000 Talents Program for Young Scholars from the Chinese Government,
and the Chinese Academy of Sciences (CAS) President's International 
Fellowship Initiative No. 2010Y2GB03.
Additional support was received by
the CAS and the Institute of Software of the CAS.
Partial support was also received from a Marsden grant of New Zealand 
and the China Basic Research Program (973) grant No.~2014CB340302.}}
\author{George Barmpalias  \and Andrew Lewis-Pye}
\date{\today}
\begin{document}
\maketitle
\begin{abstract}
Every real is computable from a \ml random real. 
This well known result in algorithmic randomness was proved by
Ku\v{c}era \cite{MR820784} and G{\'a}cs \cite{MR859105}.
In this survey article we discuss various approaches to the problem of coding
an arbitrary real into a \ml random real, and also describe new results concerning optimal methods of coding.  We start with a simple presentation of 
the original methods
of Ku\v{c}era and G{\'a}cs and 
then  rigorously 
demonstrate their limitations in terms of the size of the redundancy in the codes
that they produce. Armed with a deeper understanding of these methods, we then proceed
to motivate  and illustrate aspects of the new coding method that was 
recently introduced by
Barmpalias and Lewis-Pye in \cite{optcod} and which achieves optimal logarithmic redundancy,
an exponential improvement over the original redundancy bounds.
%
 \end{abstract}
\vfill 
\noindent{\bf George Barmpalias}\\
\noindent State Key Lab of Computer Science, 
Institute of Software, Chinese Academy of Sciences, Beijing, China.
School of Mathematics, Statistics and Operations Research,
Victoria University of Wellington, New Zealand. \\
\textit{E-mail:} \texttt{\textcolor{dnrgr}{barmpalias@gmail.com}}.
\textit{Web:} \texttt{\href{http://barmpalias.net}{http://barmpalias.net}}
\vfill
\noindent{\bf Andrew Lewis-Pye}\\
\noindent Department of Mathematics,
Columbia House, London School of Economics, 
Houghton St., London, WC2A 2AE, United Kingdom. \\
\textit{E-mail:} \texttt{\textcolor{dnrgr}{A.Lewis7@lse.ac.uk.}}
\textit{Web:} \texttt{\textcolor{dnrre}{http://aemlewis.co.uk}} 
 \vfill\thispagestyle{empty}
\clearpage
\newcommand{\upfm}{universal prefix-free machine }
\newcommand{\pfn}{prefix-free}
\newcommand{\upfmn}{universal prefix-free machine}
\newcommand{\lhs}{left-hand-side }
\newcommand{\rhs}{right-hand-side }
\newcommand{\nua}{\nu_{\ast}}
\newcommand{\sz}{$\Sigma^0_1$ }
\newcommand{\BHMN}{Bienvenu, H{\"{o}}lzl, Miller and Nies }
\newcommand{\BGKNT}{ Bienvenu, Greenberg, Ku\v{c}era, Nies and Turetsky }
\newcommand{\CDFT}{Chalcraft, Dougherty, Freiling, and Teutsch }
\newcommand{\BDM}{Barmpalias, Downey, and McInerney }

\section{Introduction}
Information means structure and regularity, while
randomness means the lack of structure and regularity.
One can formalize and even quantify this intuition in the context of
algorithmic randomness and Kolmogorov complexity, where the interplay between
information and randomness has been a principal driving 
force for much of the
research. 
\begin{equation}\label{vp1ODAZwOY}
\parbox{12cm}{How much information can be coded into a random binary sequence?}
\end{equation}
This question has various answers, depending on how it is formalized, but as we are going to see
in the following discussion, for sufficiently strong randomness the answer is `not much'.

\subsection{Finite information}
In the case of a finite binary sequence (string) $\sigma$,  let $K(\sigma)$ denote the
\pf complexity of $\sigma$. Then $\sigma$ is  $c$-incompressible
if $K(\sigma)\geq |\sigma|-c$. Here we view the underlying optimal universal \pf machine $U$
as a {\em decompressor} or {\em decoder}, which takes a string/program $\tau$ and may
output another string $\sigma$, in which case $\tau$ is regarded as a description of $\sigma$.
Then $K(\sigma)$ is the length of the shortest description of $\sigma$ and 
the random strings are the $c$-incompressible strings for some $c$, which is known
as the {\em randomness deficiency}. It is well known that the shortest description of a string
is random, \ie 
there exists a constant $c$ 
such that each shortest description 
is $c$-incompressible.
In other words,
\begin{equation}\label{dp6vdN9peL}
\parbox{11.5cm}{every string $\sigma$ 
can be coded into a random string (its shortest description),
of length the Kolmogorov complexity of $\sigma$}
\end{equation}
which may seem as a strong positive answer to Question
\eqref{vp1ODAZwOY}, in the sense that every 
string $\sigma$ can be coded into a random string. The following proposition, however, points in the opposite direction: 
\begin{prop}[Folklore]\label{369JUGh32W}
If $U$ is an optimal universal \pf machine
then there exists a constant $c$ such that $U(\sigma)\un$ for all strings $\sigma$ such that
$K(\sigma)\geq |\sigma|+c$.\footnote{The proof of this fact is based on the idea that
each string in the domain of $U$ is a \pf description of itself (modulo some fixed overhead).
In other words, if $U(\sigma)\de$ then $\sigma$ can be used to describe itself, with respect to
some \pf machine that is then simulated by $U$, producing a $U$-description of $\sigma$ of
length $|\sigma|+c$ for some constant $c$.}
\end{prop}
Viewing $U$ as a universal decompressor, Proposition 
\eqref{369JUGh32W} says that a sufficiently random string cannot be decoded into anything, which
means that in that sense it does not effectively code any information. According to this fact, Question 
\eqref{vp1ODAZwOY} has a strong negative answer.

\subsection{Bennett's analogy for infinite information}
The notions and issues discussed in the previous section 
have infinitary analogues which concern coding infinite binary sequences ({\em reals})
into random reals.  For sufficiently strong (yet still moderate) 
notions of randomness for reals (such as the randomness  
corresponding to statistical tests or predictions that are definable in
arithmetic with two quantifiers), the answer to Question
\eqref{vp1ODAZwOY} is {\em not much};
such random reals cannot solve the 
halting problem or even compute a complete extension of Peano
Arithmetic. 
Charles Bennett (see \cite{Bennett1988}) asked if Question
\eqref{vp1ODAZwOY} can have a strongly positive answer, just as in the finite case,
for a standard notion of algorithmic randomness such as \ml randomness.
Remarkably, Ku\v{c}era \cite{MR820784} and G{\'a}cs \cite{MR859105} 
gave a positive answer to Bennett's question.
\begin{thm}[\KG theorem]\label{BZG8etR5Er}
Every real is computable from a \ml random real.
\end{thm}
Bennett  \cite{Bennett1988} commented: 
\begin{equation*}
\parbox{14cm}{{\em ``This is the infinite analog of the far more obvious fact that every 
finite string is computable from an algorithmically random string (e.g. its minimal program).''}
}
\end{equation*}
Here we argue that Bennett's suggested analogy between 
\eqref{dp6vdN9peL} and Theorem 
\ref{BZG8etR5Er}
is not precise, in the sense that 
it misses the quantitative aspect of \eqref{dp6vdN9peL} -- namely that
the random code can be chosen short (of length the complexity of the string). 
It is much easier to code $\sigma$ into a random string which is much longer than $\sigma$,
than code it into a random string of length at most $|\sigma|$. The
analogue of  `{\em length of code}' for infinite codes, is the {\em use-function} 
in a purported Turing reduction underlying the computation of a real $X$ from a random real $Y$.
The use function for the reduction is a function $f$ such that for each $n$, the first $n$ bits of $X$ can be uniformly computed
from the first $f(n)$ bits of $Y$. 

\subsection{A quantitative version of the \texorpdfstring{\KG}{KG} theorem?}
The more precise version of Bennett's suggested analogy that we have just discussed
can be summarized in Table \ref{ABuEidRZu}, where $\sigma^{\ast}$ denotes the shortest program for $\sigma$.\footnote{If
there are several shortest strings $\tau$ such that $U(\tau)=\sigma$ then $\sigma^{\ast}$ denotes the one
that converges the fastest.}
So what is the analogue of the code length in the \KG theorem? If we code a real $X$ into a \ml random real
$Y$, how many bits of $Y$ do we need in order to compute the first $n$ bits of $X$? This question has been discussed in the literature
(see below) but, until recently, only very incomplete answers were known. 
Ku\v{c}era \cite{MR820784} did not provide tight calculations
and various textbook presentations of the theorem 
(\eg Nies \cite[Section 3.3]{Ottobook}) estimate the use-function in this 
reduction of $X$ to a \ml random $Y$ to be of the order $n^{2}$. In fact, the actual bound that can be obtained by 
Ku\v{c}era's method is $n\log n$. G\'{a}cs used a more elaborate argument and obtained
the upper bound $n+\sqrt{n}\cdot\log n$, which is $n+\smo{n}$, and the same bound was also obtained later by
Merkle and Mihailovi\'{c} \cite{jsyml/MerkleM04} who used an argument in terms of supermartingales. 

\subsection{Coding into random reals, since \texorpdfstring{Ku\v{c}era and G\'{a}cs}{KG}}
The \KG coding method has been combined with various arguments in order to produce
\ml random reals with specific  computational properties.
The first application already appeared in \cite{Kucera:87}, where a high incomplete \ml random real
computable from the halting problem was constructed.
Downey and Miller \cite{MR2170569} and later
Barmpalias, Downey and Ng \cite{BDNGP} 
presented a variety of different versions of this method, which allow some control over the 
degree of the random real
which is  coded into. 
Doty \cite{cie/Doty06} revisited the \KG theorem from the viewpoint of constructive dimension.
He characterized the asymptotics of the redundancy in computations of an infinite sequence
$X$ from a random oracle in terms of the constructive dimension of $X$.
We should also mention that this is not the only method for coding
into members of a positive measure \pz class (or into the class of \ml random reals).
Barmpalias, Lewis-Pye and Ng \cite{BLNg08} used a different method in order to show that
every degree that computes a complete extension of Peano Arithmetic is the supremum of
two \ml random degrees. 

It is fair to say that all of these methods rely heavily on the density of reals inside a nonempty \pz
class that consists entirely of \ml reals.
This is also true of more recent works such as
\BGKNT \cite{coher}, Day and Miller \cite{daymiller15} and
Miyabe, Nies and Zhang \cite{miyabe15}. Khan \cite{KhanLDP} explicitly studies 
 the properties of density inside \pz classes, not necessarily consisting entirely of \ml random reals.
Much of this work is concerned with lower bounds on the density that a \ml real has inside
every \pz class that contains it. In our analysis of the \KG theorem we isolate the role of density
in the argument and show that, in a sense,  tighter oracle-use in computations from \ml random oracles
is only possible through methods that do not rely on such density requirements. 

\begin{table}
\colorbox{black!5}{\arrayrulecolor{green!50!black} 
  \begin{tabular}{lccc}
{\small\em Notion}\hspace{2cm} &{\small\em Finite} 
&\hspace{0.2cm}  &  {\small\em Infinite}  \\[0.1ex]\cmidrule[0.5pt]{1-4}
{\small Source} &{\small $\sigma$} &	\hspace{0.2cm}  & {\small $X$} \\[1ex]
 {\small Code} &{\small $\sigma^{\ast}$}	& \hspace{0.2cm}  &  {\small $Y$}\\[1ex]
 {\small Code-length} &{\small $|\sigma^{\ast}|$} & \hspace{0.2cm} & {\small $n\mapsto f(n)$}\\[1ex]
 {\small Optimal code} &{\small $K(\sigma)$} & \hspace{0.2cm} & {\small ?}\\[1ex]
\end{tabular}}\centering
\caption{Quantitative analogy between finite and infinite codes; here $n\mapsto f(n)$ refers to an `optimal' non-decreasing upper bound
on the use-function in the computation of $X$ from $Y$.}\label{ABuEidRZu}
\end{table}

\section{Coding into an effectively closed set subject to density requirements}
The arguments of 
Ku\v{c}era and G\'{a}cs both
provide a method for coding  an arbitrary real $X$ into a member of an effectively closed set $P$
(a \pz class),  and rely on certain density requirements for the set of reals $\mathcal{P}$.
The connection to Theorem \ref{BZG8etR5Er} is that
the class of \ml random reals is the effective union of countably many \pz classes of positive measure.
The only difference in the two methods is that Ku\v{c}era codes $X$ one-bit-at-a-time
(with each bit of $X$ coded into a specified segment of $Y$) while G\'{a}cs codes $X$
block-by-block into $Y$, with respect to a specified segmentation of $X$.

\subsection{Overview of the \texorpdfstring{\KG}{KG} argument}

In general, we can code $m_i$ many bits of $X$ at the $i$th coding step, using a block in $Y$ of length $\ell_i$,
as Table \ref{LmWO8NSr8P} indicates. 
We leave the parameters $(m_i), (\ell_i)$ unspecified for now, while in the following it will become clear what the growth of 
this sequence needs to be in order for the argument to work.
Note that  the bit-by-bit version of the coding is
the special case where $m_i=1$ for all $i$. The basic form of the coding process (which we shall elaborate on later) can be outlined as follows. 
\begin{enumerate}[\hspace{0.5cm}(1)]
\item Start with a  \pz class $\mathcal{P}\neq\emptyset$ which only contains (Martin-L\"{o}f) randoms.
\item Choose the {\em length $m_i$ of  the block} coded at step $i$.
\item Choose the length $\ell_i=m_i+g(i)$ used for coding the $i$th block.
\item The {\em oracle-use} for the first $M_n=\sum_{i<n} m_i$ bits is $L_n=\sum_{i< n} \ell_i$.
\item Form a subclass  $\mathcal{P}^{\ast}$ of  $\mathcal{P}$ with the property that
for all but finitely many $n$ and for each $X\in\mathcal{P}^{\ast}$, there are at least $2^{m_n}$ extensions of 
$X\restr_{L_n}$ of length $L_{n+1}$ which have infinite extensions in $\mathcal{P}^{\ast}$.
\item Argue that $\mathcal{P}^{\ast}\neq\emptyset$  \ \ (due to the growth of $(\ell_i)$, relative to $(m_i)$).
\end{enumerate}

A crucial fact here is that if  $\mathcal{P}$ is a \pz class then  $\mathcal{P}^{\ast}$ is also a \pz class.
In Section \ref{n3ozZCJarz} we turn this outline into a modular proof, which makes the required properties of
the parameters $(m_i), (\ell_i)$  transparent. 
We will show that apart from the computability of 
 $(m_i), (\ell_i)$,
the following facts characterize the necessary and sufficient constraints on the two sequences 
for the coding to work.
\begin{enumerate}[(i)]
\item If $\sum_i 2^{m_i-\ell_i}<\infty$
then there exists a \pz class of positive measure that consists entirely 
of \ml random reals such that $\mathcal{P}^{\ast}\neq \emptyset$;
\item If $\sum_i 2^{m_i-\ell_i}=\infty$ and
$\mathcal{P}$ is a \pz class such that 
$\mathcal{P}^{\ast}\neq\emptyset$ then  $\mathcal{P}$ contains a real which is not \ml random.
\end{enumerate}

\begin{table}
\colorbox{black!5}{\arrayrulecolor{green!50!black} 
  \begin{tabular}{ll}
\cmidrule[0.5pt]{1-2}
{\small $m_i$} &{\small Length of the $i$th block of $X$} \\[1ex]
{\small $\ell_i$} &{\small Length of the $i$th block of $Y$} \\[1ex]
{\small $M_n$} &{\small Number of bits of $X$ coded after $n$-many coding steps: $M_n:=\sum_{i<n} m_i$} \\[1ex]
{\small $L_n$} &{\small Length of $Y$ used in the computation of $X\restr_{M_n}$: $L_n:=\sum_{i<n} \ell_i$} \\[1ex]
\cmidrule[0.5pt]{1-2}
\end{tabular}}\centering
\caption{Parameters of the \KG coding of $X$ into $Y$}\label{LmWO8NSr8P}
\end{table}

\subsection{The general \texorpdfstring{\KG}{KG} argument}\label{n3ozZCJarz}
We give a modular argument  in terms of \pz classes, showing that every real is computable from a \ml random real,
and consisting of a few simple lemmas. 
We use Martin-L\"{o}f's paradigm of algorithmic randomness, much like in the original argument of
Ku\v{c}era and G\'{a}cs.\footnote{However our presentation has been significantly assisted by 
Merkle and Mihailovi\'{c}
\cite{jsyml/MerkleM04}, who phrased the argument in terms of martingales.}
In the next definition, recall that for finite $\sigma$, $[\sigma ]$ is the set of all infinite extensions of $\sigma$. 

\begin{defi}[Extension property]\label{kvWs5nRqzo}
Given a \pz class $P$ and sequences $(m_i)$, $(\ell_i)$ of positive integers, let 
$M_n:=\sum_{i<n} m_i$, $L_n:=\sum_{i<n} \ell_i$ and say that
$P$ has the extension property with respect to
$(m_i)$, $(\ell_i)$ if  for each $i$, every string $\sigma$ of length $L_i$ with 
$[\sigma]\cap P\neq\emptyset$ has at least $2^{m_i}$ extensions $\tau$ of length $L_{i+1}$ 
such that $P\cap [\tau]\neq\emptyset$.
\end{defi}

The first lemma says that subject to certain density conditions on a \pz class $P$, 
every real is computable from a member of $P$.

\begin{table}
\colorbox{black!5}{\arrayrulecolor{green!50!black} 
  \begin{tabular}{cl}
\cmidrule[0.5pt]{1-2}
{\small $\ell_i-m_i$} &\hspace{0.3cm}{\small Overhead at the $i$th coding step}\\[1ex]
{\small $\sum_{i<n}(\ell_i-m_i)$} &\hspace{0.3cm}{\small Accumulated overhead after $n$ coding steps}\\[1ex]
{\small $\sum_{i} 2^{m_i-\ell_i}<\infty $} &\hspace{0.3cm}{\small Necessary and sufficient condition for successful coding}\\[1ex]
\cmidrule[0.5pt]{1-2}
\end{tabular}}\centering
\caption{Overheads in the \KG coding of $X$ into $Y$}\label{1pDc2J1qXg}
\end{table}

\begin{lem}[General block coding]\label{NrfxnwPq}
Suppose that $P$ is a \pz class, and $(m_i)$, $(\ell_i)$ are 
computable sequences of positive integers. If $P$ has the extension property with respect to
$(m_i)$, $(\ell_i)$, then 
every sequence is computable from a real in $P$ with use $L_{s+1}$
for bits in $[M_s, M_{s+1})$.
\end{lem}
\begin{proof}
For any string $\sigma$ of length $L_i$ consider
the variables $w_0(\sigma)[s],\dots w_{2^{m_i}-1}(\sigma)[s]$ for strings, 
which are defined dynamically according to
the approximation $(P_s)$ to $P$ as follows. At stage 0 let
$w_j(\sigma)[0]\un$ for all $j<2^{m_i}$. At stage $s+1$ find the least 
 $t<2^{m_i}$ such that one of the following holds:
\begin{enumerate}[\hspace{0.5cm}(a)]
\item  $w_t(\sigma)[s]\un$;
\item  $w_t(\sigma)[s]\de$ and $[w_t(\sigma)[s]]\cap P_{s+1}=\emptyset$.
\end{enumerate}
In case (a) look for the lexicographically least $\ell_i$-bit extension $\tau$ of $\sigma$
such that $[\tau]\cap P_{s+1}\neq \emptyset$ and
$w_j(\sigma)[s]\not\simeq \tau$ for all $j<2^{m_i}$. If no such exists, terminate
the process (hence let $w_j(\sigma)[n]\simeq w_j(\sigma)[s]$ for all $j<2^{m_i}$ and all $n>s$).
Otherwise define $w_t(\sigma)[s+1]=\tau$  and go to the next stage.
In case (b) let $w_t(\sigma)[s+1]\un$ and go to the next stage.

By the hypothesis of the lemma, for every $i$ and every string $\sigma$
of length $L_i$ such that $[\sigma]\cap P\neq\emptyset$, the words
$w_j(\sigma)[s]$, $j<2^{m_i}$ reach limits $w_j(\sigma)$ after finitely many stages,
such that: 
\begin{itemize}
\item  $j\neq k\Rightarrow w_j(\sigma)\neq w_k(\sigma)$ for all $j,k<2^{m_i}$;
\item  $[w_j(\sigma)]\cap P\neq\emptyset$.
\end{itemize}
Consider the Turing functional $\Phi$ which, given oracle $Y$, works inductively as follows.
Suppose that $\Phi(Y\restr_{L_i})\restr_{M_i}$ has been calculated. The functional then searches for the least pair $(j,s)$
(under a fixed effective ordering of all pairs, of order type $\omega$) 
such that $j<2^{m_i}$, $w_j(Y\restr_{L_i})[s]\de$
and  is a prefix of $Y$. 
For $\tau$ which is the $j$th string of length
$m_i$ (under the lexicographical ordering) 
the functional then defines 
$\Phi(Y\restr_{L_i+\ell_i})=\Phi(Y\restr_{L_i})\ast \tau$.
By  construction $\Phi$ is consistent, and if $\Phi(Y\restr_{L_i})$ is defined it has length $M_i$.
Finally we show that $\Phi$ is onto the Cantor space. Given $X$ we can inductively construct $Y$
such that $\Phi(Y)=X$. Suppose that we have constructed $Y\restr_{L_i}$ such that
$\Phi(Y\restr_{L_i})=X\restr_{M_i}$ and $Y\restr_{L_i}$ is extendible in $P$. Let $\sigma$ be the 
unique string of length $m_i$ such that $X\restr_{M_i}\ast \sigma$ is a prefix of $X$.
Then $w_j(Y\restr_{L_i})$ is defined for all $j<2^{m_i}$ and takes distinct values for different 
$j$. Let $t$ be the index of $\sigma$
in the lexicographical ordering of strings of length $m_i$.
Then let $Y\restr_{L_{i+1}}=w_t(Y\restr_{L_i})$. Clearly
$Y\restr_{L_{i+1}}$ is extendible in $P$ and moreover
$\Phi(Y\restr_{L_{i+1}})=X\restr_{M_{i+1}}$. This completes the induction step
in the construction of $Y$ and shows that
$\Phi(Y)=X$.
\end{proof}

\begin{figure}
\scalebox{0.8}{
\begin{tikzpicture}[
Nnode/.style={rectangle,  minimum size=6mm, very thick,
draw=white!70!black!10, 
top color=white!50!black!10, bottom color=white!50!black!10,
rounded corners,  font=\small },
tnode/.style={rectangle,  minimum size=6mm, rounded corners, very thick,
draw=red!50!black!50, top color=white, bottom color=red!50!black!20, 
 font=\small},
 ttnode/.style={rectangle,  minimum size=6mm, rounded corners, very thick,
draw=yellow!50!black!50, top color=white, bottom color=yellow!50!black!20, 
 font=\small}]
 \node (cr) [Nnode, outer sep=3pt, inner sep=7pt] at (1,1.5) {\parbox{5cm}{Fast-growing overheads $\ell_i-m_i$}};
\node (ivr) [Nnode, outer sep=3pt, inner sep=7pt] at (7,1.5) {\parbox{3.5cm}{Density property in $P$}};
\node (fvr) [Nnode, outer sep=3pt, inner sep=7pt] at (7,0) {\parbox{4cm}{Extension property in $P$}};
\node (svr) [Nnode, outer sep=3pt, inner sep=7pt] at (1,0) {\parbox{3.5cm}{Successful coding in $P$}};

\draw [->] (cr) -- (ivr); 
\draw [->] (ivr) -- (fvr); 
\draw [->] (fvr) -- (svr);
\end{tikzpicture}}
\centering
\caption{{\textrm Diagrammatic representation of the \KG coding argument.}}
\label{fig:dynam2}
\end{figure}
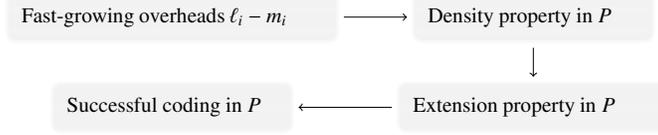

Recall that for $\sigma$ of length $n$,  the $P$-density of $\sigma$  is defined to be $2^n \cdot \mu([\sigma]\cap P)$, where $\mu$ denotes Lebesgue measure on Cantor space.  

\begin{defi}[Density property]
Given  $P$, $(m_i)$, $(\ell_i)$ as in Definition \ref{kvWs5nRqzo}
we say that
$P$ has the density property with respect to
$(m_i)$, $(\ell_i)$ if  
for each $n$, every string of length $L_n$ with 
$[\sigma]\cap P\neq\emptyset$ has $P$-density at least $2^{m_n-\ell_n}$.
\end{defi}

\begin{lem}[Density and extensions]\label{Wc6ZwibuvM}
Given  $P$, $(m_i)$, $(\ell_i)$ as in Definition \ref{kvWs5nRqzo}, if
$P$ has the density property with respect to
$(m_i)$, $(\ell_i)$ then it also has the
extension property with respect to
$(m_i)$, $(\ell_i)$.
\end{lem}
\begin{proof}
This follows from the general fact that if 
the $P$-density of $\sigma$ is at least $2^{-t}$ for some $t$,
then given any $m$, there are at least $2^m$ extensions $\tau$ of $\sigma$ of length $|\sigma|+t+m$ such that
$[\tau]\cap P\neq \emptyset$. In order to prove the latter fact, suppose for a contradiction that it is not true.
Then the $P$-density of  $\sigma$ would be at most $(2^{m}-1)\cdot 2^{-m-t}=2^{-t}-2^{-m-t}<2^{-t}$
which contradicts the hypothesis.
\end{proof}

\begin{lem}[Lower bounds on the density]\label{cn1lvqEMtb}
Let $P$ be a \pz class and let $(m_i),(\ell_i)$ be computable sequences of positive integers such that
$\sum_{i} 2^{m_i-\ell_i}<\mu(P)$. Then there exists a \pz class $P^{\ast}\subseteq P$ 
which has the extension property with respect to $(m_i),(\ell_i)$.
\end{lem}
\begin{proof}
We construct a \sz class $Q$ in stages and let $(P_s)$ be a \pz approximation to $P$, where
each $P_s$ is a clopen set.
A string $\sigma$
is {\em active} at stage $s+1$ if  
it is of length $L_n$ for some $n$ and
$[\sigma]\cap (P_s-Q_s)\neq \emptyset$.
Moreover $\sigma$ of length $L_n$ {\em requires attention} at stage $s+1$ if it is active
at this stage and the $(P_s-Q_s)$-density of $\sigma$ is at most $2^{m_n-\ell_n}$.
At stage $s+1$, we pick the least string of length $<s$ which requires attention (if such exists)
and enumerate $[\sigma]\cap (P_s-Q_s)$ into $Q$. If this enumeration occurred, we say that
the construction acted on string $\sigma$ at stage $s+1$.
This concludes the construction.

First we establish an upper bound on the measure of $Q=\cup_s Q_s$.
Clearly the construction can act on a string at most once. The measure that is added to $Q$
at stage $s+1$ if the construction acts on $\sigma$ of length $L_n$ at this stage, is at most 
$2^{-L_n+m_n-\ell_n}$. Therefore the total measure enumerated into $Q$ throughout the
construction is bounded above by:
\[
\sum_{n} \sum_{\sigma\in 2^{L_n}} 2^{-L_n+m_n-\ell_n}=
\sum_{n} 2^{L_n}\cdot 2^{-L_n+m_n-\ell_n}=
\sum_{n} 2^{m_n-\ell_n} < \mu(P).
\]
It follows that  $P^{\ast}:=P-Q$ is a nonempty \pz class, and
by the construction we have that for every $n$ and every 
string $\sigma$ of length $L_n$, if $[\sigma]\cap P^{\ast}\neq\emptyset$ then
the $P^{\ast}$-density of $\sigma$ is at least $2^{m_n-\ell_n}$.
By Lemma \ref{Wc6ZwibuvM} this means that  every $P^{\ast}$-extendible
string of length $L_n$ for some $n$  has at least $2^{m_n}$ many $P^{\ast}$-extendible
extensions of length $L_n+m_n-(m_n-\ell_n)=L_{n+1}$. Hence $P^{\ast}$ has the extension property 
with respect to $(m_i)$, $(\ell_i)$.
\end{proof}

\begin{coro}[General block coding]\label{NrfxnwPe}
Suppose that $P$ is a \pz class, and $(m_i)$, $(\ell_i)$ are 
computable sequences of positive integers. If 
$\sum_{i} 2^{m_i-\ell_i}<\mu(P)$
 then  every sequence is computable from a real in $P$ with use $L_{s+1}$
for bits in $[M_s, M_{s+1})$.
\end{coro}
\begin{proof}
By Lemma \ref{cn1lvqEMtb}
we can consider a \pz class $P^{\ast}\subseteq P$
which has the extension property 
with respect to $(m_i)$, $(\ell_i)$. The statement then follows by Lemma \ref{NrfxnwPq}
and the fact that $P^{\ast}\subseteq P$.
\end{proof}

Note that, while Corollary \ref{NrfxnwPe} seems to require (a) $\sum_{i} 2^{m_i-\ell_i}<\mu(P)$, if $P$ is of positive measure then the condition (b) $\sum_{i} 2^{m_i-\ell_i}<\infty$ suffices to ensure that  $\sum_{i\geq d} 2^{m_i-\ell_i}<\mu(P)$ for some $d$ -- meaning that (b) is sufficient to give the existence of the required functional (albeit with some added non-uniformity required in specifying the index of the reduction).

\subsection{The oracle-use in the general \texorpdfstring{\KG}{KG} coding argument}

Recall that if $X$ can be computed from $Y$ with  the use function on argument $n$ bounded by $n+g(n)$, then we say that $X$ can be computed from $Y$ with \emph{redundancy} $g(n)$. 
Note that in the following corollary we do not need to require that $h,h_r$ are computable.
\begin{coro}\label{AKkRF5WHu}
Suppose $(m_i)$, $(\ell_i)$ are 
computable sequences of positive integers with  
$\sum_{i} 2^{m_i-\ell_i}<1$ and suppose 
$h,h_r$ are nondecreasing functions such that: 
\[
\sum_{i\leq s} \ell_i \leq h\left(1+\sum_{i<s} m_i\right)
\hspace{0.5cm}\textrm{and}\hspace{0.5cm}
m_s+\sum_{i\leq s} (\ell_i-m_i)\leq h_r\left(\sum_{i<s} m_i\right).
\]
Then if $P$ is a $\Pi^0_1$ class of positive measure,  any sequence is computable from a real in $P$ with oracle-use $h$ and redundancy $h_r$.
\end{coro}
\begin{proof}
The first claim follows directly from Corollary \ref{NrfxnwPe} and for the second, 
recall that in the same corollary, for each $s$ and each
$n\in [M_s, M_{s+1})$, the length of the initial segment of $Y$ that is used for the 
computation of $X\restr_n$ is at most
\[
L_{s+1}= M_{s}+m_s+\sum_{i\leq s} (\ell_i-m_i)\leq
n+m_s+\sum_{i\leq s} (\ell_i-m_i)\leq
n+h_r(M_s)\leq n+h_r(n)
\] 
where the second inequality was obtained from the main property assumed for $h_r$, 
and the last inequality follows from the monotonicity of $h_r$.
\end{proof}

Without yet specifying the sequences 
$(m_i)$, $(\ell_i)$, the condition $\sum_{i} 2^{m_i-\ell_i}<1$ means
that a near-optimal choice for the sequence $(\ell_i-m_i)$ is $\ceil{2\log (i+2)}$.
This means that $\sum_{i} (\ell_i-m_i)$ will be of the order $\log (n!)$ or $n\log n$.
We may now consider an appropriate choice for the sequence $(m_i)$, which 
roughly minimizes the redundancy established in Corollary \ref{AKkRF5WHu}.
For Ku\v{c}era's coding we have $m_i=1$ for all $i$ which means
that the redundancy in this type of bit-by-bit coding is 
$n\log n$ (modulo a constant).
G\'{a}cs chose the sequence
$m_i=i+1$, and the reader may verify that this growth-rate of the blocks of the coded stream
gives a near-optimal redundancy in Corollary \ref{AKkRF5WHu}.\footnote{For example the choices 
$m_i=(i+1)^2$ or $m_i=\sqrt{i+1}$ produce redundancy considerably above  G\'{a}cs'
$\sqrt{n}\cdot\log n$ upper bound.}
In this case the function $h_r(n)=\sqrt{n}\cdot \log n$ satisfies
the second displayed inequality of Corollary \ref{AKkRF5WHu} (for almost all $n$), since
$n+1 +n\log n \leq \sqrt{(n+1)n/2}\cdot \log ((n+1)n/2)$
for  almost $n$.
Hence every real is computable from a \ml random real with this redundancy, much like G\'{a}cs had observed. 

We can now intuitively understand 
how the redundancy upper bounds $n\log n$ and $\sqrt{n}\cdot \log n$, of 
Ku\v{c}era
and G\'{a}cs respectively,  are produced. 
The argument of Section \ref{n3ozZCJarz} describes a coding process where
in $n$ coding steps we code $M_n$ many bits of $X$ into $L_n$ many bits of $Y$.
The parameter $g(i):=\ell_i-m_i$ can be seen as an {\em overhead} of the $i$th coding step,
\ie the number of additional bits we use in $Y$ in order to code the next $m_i$ bits of $X$.
Moreover, Corollary \ref{AKkRF5WHu} says that these overheads are accumulated along the coding steps
and push the redundancy of the computation to become larger over time.
In particular, the number $\sum_{i<n} g(i)$ is the redundancy (total overhead accumulated) corresponding to
$n$ coding steps. Due to the condition 
$\sum_i 2^{-g(i)}<1$ in Corollary \ref{AKkRF5WHu}
a representative choice for $g$ is $2\log (n+1)$, which means that 
$\sum_{i<n} g(i)$ needs to be of the order
$\log (n!)$ or (by Stirling's formula) $n\log n$. 

In the case of Ku\v{c}era's argument, $n$ bits of $X$ are coded in $n$ coding steps, so 
the redundancy for the computation of $n$ bits of $X$ from $Y$ following 
Ku\v{c}era's argument is of the order $n\log n$. If we are free to choose $(m_i)$, note that
a fast-growing choice will make the accumulated overhead smaller (since the coding steps for any initial segment of $X$
become less) but a different type of overhead, namely the parameter $m_s$ in the second inequality of Corollary \ref{AKkRF5WHu},
pushes the redundancy higher. G\'{a}cs' choice of $m_i=i+1$ means that in $n$ coding steps there are
$\sum_{i\leq n} m_i \approx n^2$ many bits of $X$ coded into $Y$. Hence the coding of $X\restr_n$ requires roughly
$\sqrt{n}$ coding steps, which accumulate a total of $\sqrt{n}\cdot \log \sqrt{n} \approx \sqrt{n}\cdot \log n$ in overheads
according to the previous discussion. For this reason, G\'{a}cs' redundancy is of the order $\sqrt{n}\cdot \log n$. 
We may observe that
in  G\'{a}cs' coding, the length of the next coding block $m_{n+1}$ is both:
\begin{enumerate}[(a)]
\item  the number of coding steps performed so far;
\item roughly equal to the accumulated overhead from the coding steps
performed so far.
\end{enumerate}

\subsection{Some limits of the \texorpdfstring{\KG}{KG} method}\label{eKNwQobQTs}
 
 In this section we will frequently identify a set of finite strings $V$ with the $\Sigma^0_1$ class specified by $V$, i.e.\ the set of infinite sequences extending elements of $V$. In the following proof we use the notation
$\mu_{\sigma}(C)$ for a string $\sigma$ and a set of reals $C$, which is the measure of $C$ relative
to $[\sigma]$. More precisely  $\mu_{\sigma}(C)=\mu(C\cup [\sigma])\cdot 2^{|\sigma|}$.

\begin{lem}\label{iYnW9wWoGI}
Let $P$ be a \pz class, $g$ a computable function taking positive values, such that $\sum_i 2^{-g(i)}=\infty$. Let $(n_i)$ be a computable sequence such that $n_{i+1}> n_i+g(i)$ for all $i$.
If 
\[
\parbox{13cm}{$(U_i)$ is a uniformly \ce sequence with $U_i\subseteq 2^i$ and  $\mu (P\cap U_i)<2^{-i}$ for all $i$}
\]
then every  \ml random real $X\in \cap_i (P\cap U_i)$ has a prefix in some in $U_{n_t}$ with 
$P$-density at most $2^{-g(t)}$.
\end{lem}
\begin{proof}
We define a uniform sequence $(V_i)$ of \sz classes such that 
$V_{t}\supseteq V_{t+1}$ for all $t$, inductively as follows.
Let $V_0$ (as a set of finite strings) consist of all the strings of length $n_0$. Assuming that $V_t$ has been defined,
 for each $\sigma\in V_t$ define
\[
V_{t+1}\cap [\sigma]= \left(U_{n_{t+1}}\cap [\sigma]\right)^{\left[\leq 2^{-|\sigma|}\cdot (1-2^{-g(t)-1})\right]},
\]
where for any real $r$ and any \sz class $C$ with an underlying
 computable enumeration $C[s]$ the
expression $C^{[\leq r]}$ denotes the class $C[s_{\ast}]$ where 
$s_{\ast}$ is the largest stage $s$ such that $\mu(C[s])\leq r$ if such a stage exists, 
and $s_{\ast}=\infty$ otherwise (in which case we let $C[\infty]=C$).
Clearly for each $t$ the set $V_t$ consists of strings of length $n_t$.
Then for each $t$ we have $\mu(V_{t+1})\leq (1-2^{-g(t)-1})\cdot \mu(V_{t})$ so
\[
\mu(V_{t+1})\leq \prod_{i=0}^{t} \left(1-2^{-g(i)-1}\right).
\]
By hypothesis, $\sum_i 2^{-g(i)}=\infty$ so $\prod_{i=0}^{\infty} (1-2^{-g(i)-1})=0$.
Since $g$ is computable, there exists a computable increasing sequence $(k_i)$ such that
$\prod_{i=0}^{k_t} (1-2^{-g(i)-1})<2^{-t}$ for all $t>0$.
Hence $(V_{k_i})$ is a \ml test. Now let $X$ be a  \ml random real with $X\in \cap_i (P\cap U_i)$, as in the statement of the lemma. Since $X$ is \ml random,  
$X\notin \cap_i V_{k_i}=\cap_i V_{i}$ and  there exists a maximum $t$
such $X$ has a prefix $\sigma$ in $V_t$. By the maximality of $t$ we have
$X\not\in V_{t+1}$ and since
$X\in U_{n_{t+1}}$ we must have 
$\mu_{\sigma}(U_{n_{t+1}})>1-2^{-g(t)-1}$, because otherwise
a prefix of $X$ would enter $V_{t+1}$.
Also $\mu_{\sigma}(P\cap U_{n_{t+1}})\leq 
2^{|\sigma|}\cdot \mu(P\cap U_{n_{t+1}})\leq 2^{|\sigma|-n_{t+1}}$. 
Since $\sigma\in V_t$, the length of $\sigma$ is $n_t$. Since $n_{t+1}> n_t+g(t)$ we have
$\mu_{\sigma}(P\cap U_{n_{t+1}})\leq 2^{-g(t)-1}$. From the fact that 
\[
\mu_{\sigma}(P)+ \mu_{\sigma}(U_{n_{t+1}})-\mu_{\sigma}(P\cap U_{n_{t+1}})\leq 1
\]
we can deduce that $\mu_{\sigma}(P)\leq 2^{-g(t)}$. Since $\sigma$ is a prefix of $X$ of
length $n_t$, this concludes the proof.
\end{proof}

\begin{coro}
Suppose that  $(m_i), (\ell_i)$
are computable sequences of positive integers with $\sum_i 2^{m_i-\ell_i}=\infty$. Then every
\pz class consisting entirely of \ml random reals, 
which has the density property  with respect to $(m_i), (\ell_i)$,  is empty.
\end{coro}
\begin{proof}
We apply Lemma \ref{iYnW9wWoGI}
with $n_k=L_k=\sum_{i<k} \ell_i$ and $g(i)=\ell_i-m_i$.
First note that $n_{k+1}=n_k+\ell_k>n_k+g(k)$
because $g(k)<\ell_k$, so the hypothesis of 
Lemma \ref{iYnW9wWoGI} for $(n_i)$ holds.
Second, for each $i$ let $\sigma_i^{\ast}$ be the leftmost $P$-extendible string of
length $i$ and let $U_i$ consist of $\sigma_i^{\ast}$ as well as the strings of length $i$
which are lexicographically to the left of $\sigma_i^{\ast}$.  Then $(U_i)$ is uniformly \ce and
$\mu(P\cap U_i)=\mu(P\cap [\sigma_i^{\ast}])\leq 2^{-i}$ for all $i$.
Now suppose that $P$ is non-empty and consider the leftmost path $X$ through $P$. By our assumptions regarding  $P$, the real
$X$ is \ml random, so by Lemma \ref{iYnW9wWoGI}  there exists some $k$ such that 
the $P$-density of $X\restr_{L_k}$ is less than $2^{m_k-\ell_k}$. This means that
there is a $P$-extendible string of length $L_k$ 
with $P$-density below $2^{m_k-\ell_k}$, so $P$ does not have the 
 density property  with respect to $(m_i), (\ell_i)$.
\end{proof}

\begin{coro}[Lower bounds on the density inside a \pz class of random reals]\label{EOb4nucsX}
Let $P$ be a nonempty \pz class consisting entirely of \ml random reals, let $g$ be a computable function,
 and let
$(L_i)$ be an increasing sequence of positive integers such that $L_{t+1}> L_t+g(t)$ for all $t$. 
Then the the following are equivalent:
\begin{enumerate}[\hspace{0.5cm}(a)]
\item For every $i$ the $P$-density of any $P$-extendible string of length $L_i$ is $\Omega(2^{-g(i)})$
\item $\sum_i 2^{-g(i)}<\infty$
\end{enumerate}
where the asymptotic notation $\Omega(2^{-g(i)})$ means $\geq 2^{-g(i)-c}$ for some constant $c$.
\end{coro}

\section{Coding into randoms without density assumptions}
In \cite{optcod} a new coding method was introduced which allows for coding every real into a \ml
random real with optimal, logarithmic redundancy. 
We call this method {\em density-free coding} as it does not rely on density assumptions
inside \pz classes, which is also the reason why it gives an exponentially better redundancy upper bound.

\begin{lem}[Density-free coding, from \cite{optcod}]\label{Dedk3zkhlL}
Let $(u_i)$ be a nondecreasing computable sequence and let $\PP$ be a \pz class.
If $\sum_i 2^{i-u_i}<\mu(\PP)$ then every binary stream is uniformly computable from
some member of $\PP$ with oracle-use $(u_i)$.
\end{lem}

Note that by letting $P$ be a \pz class of \ml random reals of sufficiently  large measure, 
Lemma \ref{Dedk3zkhlL} shows that every real is computable from a \ml random real with use
$n+2\log n$, \ie with logarithmic redundancy. On the other hand in \cite{bakugaopt} it was shown that
this is optimal, in the sense that if $\sum_i 2^{i-u_i}=\infty$ then there is a real which is not computable
from any \ml random real with use $n\mapsto u_n$. In particular, given a real $\epsilon$, 
redundancy $\epsilon\cdot\log n$ in a computation from a random oracle is possible
for every real if and only if $\epsilon>1$.

We shall not give a proof of Lemma \ref{Dedk3zkhlL}. Instead, 
we will discuss some aspects of this more general coding method, which contrasts the
more restricted \KG coding whose limitations we have already explored.

\subsection{Coding as a labelling task}
Coding every real into a path through a tree $\TT$ in the Cantor space 
involves constructing a Turing functional $\Phi$ which is onto the Cantor space, even
when it is restricted to $\TT$. 
In fact, this is normally done in such a way that there is a subtree $\TTast$ of $\TT$
such that  $\Phi$ is a bijection between $[\TTast]$ and $2^{\omega}$.
In this case we refer to $\TTast$ as the {\em code-tree}.
Suppose we fix $\TT$ and consider constructing a functional for which the use $u_n$ on argument $n$ does not depend upon the oracle. Then the functional $\Phi$ can be
constructed as a partial computable `labelling' of the finite branches of $\TT$. If the label $x_{\sigma}$ is placed on $\tau$, this means that $\Phi$ outputs $\sigma$ when $\tau$ is the oracle. If we also suppose that the use function is strictly increasing, then the labelling might be assumed to satisfy the following conditions:

\begin{enumerate}[\hspace{0.2cm}(1)]
\item 
only strings of lengths $u_i, i\in\Nat$ of $\TT$ can have a label;
\item 
the labels placed on strings of length $u_i$ of $\TT$ 
are of the type $x_{\sigma}$ where $|\sigma|=i$;
\item 
if  label $x_{\sigma}$ exists in $\TT$ then all labels $x_{\rho}$, $\rho\in 2^{\leq |\sigma|}$  exist in $\TT$;
\item 
each string in $\TT$ can have at most one label;
\item 
if  $\rho$ of length $u_k$ in $\TT$ has label $x_{\sigma}$ then
for each $i<k$, $\rho\restr_{u_i}$ has label $x_{\sigma\restr_i}$. 
\end{enumerate}

 The reader may take a minute to view the Ku\~{c}era  coding as detailed in Section \ref{n3ozZCJarz}
 as a labelling satisfying the properties (1)-(5) above.
It is clear that: 
\[
\parbox{13.5cm}{The code-tree $\TTast$ in the Ku\~{c}era  coding is isomorphic to the full binary tree.}
\]
The new coding behind Lemma \ref{Dedk3zkhlL} is also a labelling process, but in this case the code-tree
can be much more complex. 

\subsection{Fully labelable trees}
If $(u_i)$ is an increasing sequence of positive integers,
a  $(u_i)$-tree $T$ 
is a subset of $\{\lambda\}\cup(\cup_i 2^{u_i})$
which contains the
empty string and is downward closed, in the sense that for each $\sigma\in 2^{u_{i+1}}\cap T$,
the string $\sigma\restr_{u_i}$ belongs to $T$.
The elements of a $(u_i)$-tree $T$ are called nodes and the
$t$-level of $T$ consists of the nodes of $T$ of length $u_t$.
The full binary tree of height $k$ is $2^{\leq k}$ ordered by the prefix relation.
Note that a $(u_i)$-tree is a  tree, in the sense that it is a partially ordered 
set (with respect to the prefix relation) in which the predecessors of each member are linearly ordered.
Hence given any $k\in\Nat$, we may talk about  a $(u_i)$-tree being isomorphic to the full binary tree 
of height $k$. 
When we talk about two trees being isomorphic, 
it is in this sense that we shall mean it -- as partially ordered sets. 
A labelling of a $(u_i)$-tree is
a partial map from the nodes of the tree to the set of labels
$\{x_{\sigma}\ |\  \sigma\in 2^{<\omega}\}$ 
which satisfies properties (1)-(5) of the previous section.
A full labelling of a $(u_i)$-tree is a labelling 
$\{x_{\sigma}\ |\  \sigma\in 2^{<\omega}\}$ 
with the property that for every $\sigma$ there exists
a node on the $(u_i)$-tree which has label $x_{\sigma}$.

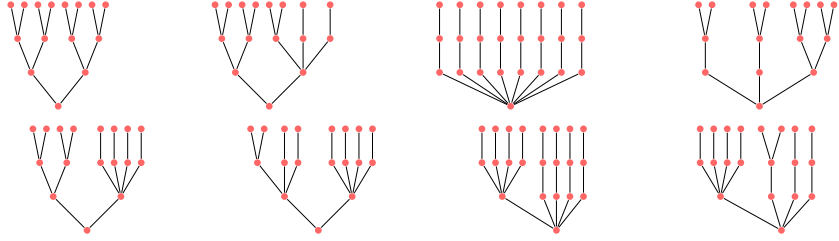
\begin{figure}
  \begin{center}
    \mbox{
      \subfigure{\scalebox{0.9}{
\begin{tikzpicture}[
    grow                    = up,
    level distance          = 0.5cm,
    level 1/.style={sibling distance=0.8cm},
   level 2/.style={sibling distance=0.4cm},
   level 3/.style={sibling distance=0.2cm},
   every node/.style={fill=red!60,circle,inner sep=1pt},
  ]
	\node  {}
		child {
			node{} 
				child {node {} 
						child {node {}}
						child {node {}}}
				child {node {} 
						child {node {}} 
						child {node {}}}
			} 
		child {
			node{} 
				child {node {} 
						child {node {}}
						child {node {}}}
				child {node {} 
						child {node {}} 
						child {node {}}}
			} ;
\end{tikzpicture}
}} \hspace{1cm}
      \subfigure{\scalebox{0.9}{
\begin{tikzpicture}[
    grow                    = up,
    level distance          = 0.5cm,
    level 1/.style={sibling distance=1cm},
   level 2/.style={sibling distance=0.4cm},
   level 3/.style={sibling distance=0.2cm},
   every node/.style={fill=red!60,circle,inner sep=1pt},
  ]
	\node  {}
		child {
			node{} 
				child {node {} 
						child {node {}}}
				child {node {} 
						child {node {}}}
				child {node {} 
						child {node {}} 
						child {node {}}}
			} 
		child {
			node{} 
				child {node {} 
						child {node {}}
						child {node {}}}
				child {node {} 
						child {node {}} 
						child {node {}}}
			} ;
\end{tikzpicture}
}} \hspace{1cm}

      \subfigure{\scalebox{0.9}{
\begin{tikzpicture}[
    grow                    = up,
    level distance          = 0.5cm,
        level 1/.style={sibling distance=0.3cm},
    every node/.style={fill=red!60,circle,inner sep=1pt},
  ]
	\node (is-root) {}
		child {node{} child  {node{} child {node{}}}} 
		child {node{} child {node{} child {node{}}}} 
		child {node{} child {node{} child {node{}}}} 
		child {node{} child {node{} child {node{}}}}
		child {node{} child {node{} child {node{}}}} 
		child {node{} child {node{} child {node{}}}} 
		child {node{} child {node{} child {node{}}}} 
		child {node{} child {node{} child {node{}}}};
\end{tikzpicture}
} }
\hspace{1cm}
      \subfigure{\scalebox{0.9}{
\begin{tikzpicture}[
    grow                    = up,
    level distance          = 0.5cm,
    level 1/.style={sibling distance=0.8cm},
   level 2/.style={sibling distance=0.4cm},
   level 3/.style={sibling distance=0.2cm},
    every node/.style={fill=red!60,circle,inner sep=1pt},
  ]
	\node (is-root) {}
		child {node {}
			child {node {} child {node {}} child {node {}}}
			child {node {} child {node {}} child {node {}}}
			} 
			child {node {} 
				child {node {} child {node {}} child {node {}}}} 
			child {node {}
				child {node {} child {node {}} child {node {}}}} ;
\end{tikzpicture}
} } 
      }
%
    \mbox{
      \subfigure{\scalebox{0.9}{
\begin{tikzpicture}[
    grow                    = up,
    level distance          = 0.5cm,
    level 1/.style={sibling distance=1cm},
   level 2/.style={sibling distance=0.4cm},
   level 3/.style={sibling distance=0.2cm},
   every node/.style={fill=red!60,circle,inner sep=1pt},
  ]
	\node  {}
		child {node{} 
				child[sibling distance=0.2cm] {node {} child {node {}}}
				child[sibling distance=0.2cm] {node {} child {node {}}}			
				child[sibling distance=0.2cm] {node {} child {node {}}}			
				child[sibling distance=0.2cm] {node {} child {node {}}}			
				} 
		child {node{} 
				child {node {} 
						child {node {}}
						child {node {}}}
				child {node {} 
						child {node {}} 
						child {node {}}}
			} ;
\end{tikzpicture}
}} \hspace{1cm}
      \subfigure{\scalebox{0.9}{
\begin{tikzpicture}[
    grow                    = up,
    level distance          = 0.5cm,
    level 1/.style={sibling distance=1cm},
   level 2/.style={sibling distance=0.4cm},
   level 3/.style={sibling distance=0.2cm},
   every node/.style={fill=red!60,circle,inner sep=1pt},
  ]
	\node  {}
		child {node{} 
				child[sibling distance=0.2cm] {node {} child {node {}}}
				child[sibling distance=0.2cm] {node {} child {node {}}}			
				child[sibling distance=0.2cm] {node {} child {node {}}}			
				child[sibling distance=0.2cm] {node {} child {node {}}}			
				} 
		child {node{} 
						child[sibling distance=0.2cm] {node {} child {node {}}}
				child[sibling distance=0.2cm] {node {} child {node {}}}			
				child {node {} 
						child {node {}} 
						child {node {}}}
			} ;
\end{tikzpicture}
}} \hspace{1cm}
      \subfigure{\scalebox{0.9}{
\begin{tikzpicture}[
    grow                    = up,
    level distance          = 0.5cm,
    every node/.style={fill=red!60,circle,inner sep=1pt},
  ]
	\node (is-root) {}
		child[sibling distance=0.2cm] {node{} child  {node{} child {node{}}}} 
		child[sibling distance=0.2cm] {node{} child {node{} child {node{}}}} 
		child[sibling distance=0.2cm] {node{} child {node{} child {node{}}}} 
		child[sibling distance=0.2cm] {node{} child {node{} child {node{}}}}
		child[sibling distance=0.4cm] {node{} 
				child[sibling distance=0.2cm] {node {} child {node {}}}
				child[sibling distance=0.2cm] {node {} child {node {}}}			
				child[sibling distance=0.2cm] {node {} child {node {}}}			
				child[sibling distance=0.2cm] {node {} child {node {}}}			
				} ;
\end{tikzpicture}
} }
\hspace{1cm}
      \subfigure{\scalebox{0.9}{
\begin{tikzpicture}[
    grow                    = up,
    level distance          = 0.5cm,
    every node/.style={fill=red!60,circle,inner sep=1pt},
  ]
	\node (is-root) {}
		child[sibling distance=0.3cm] {node{} child  {node{} child {node{}}}} 
		child[sibling distance=0.4cm] {node{} child {node{} child {node{}}}} 
		child[sibling distance=0.3cm] {node{} child {node{} child {node{}} child {node{}}}} 
		child[sibling distance=0.6cm] {node{} 
				child[sibling distance=0.2cm] {node {} child {node {}}}
				child[sibling distance=0.2cm] {node {} child {node {}}}			
				child[sibling distance=0.2cm] {node {} child {node {}}}			
				child[sibling distance=0.2cm] {node {} child {node {}}}			
				} ;
\end{tikzpicture}
} } 
      }
\end{center}
\caption{Some fully labellable $(u_i)$-trees of height 3.}\label{LYPdCu3iua}
\end{figure}

A $(u_i)$-tree is called {\em fully labelable} if there exists a full labelling of it.
Figure \ref{LYPdCu3iua} illustrates some examples of fully labelable  trees of height 3. Note that
here the nodes are binary strings (hence nodes of the full binary tree) but since
they are  nodes of a $(u_i)$-tree, they can have more than two branches.
Clearly, if $T_0\subseteq T_1$ are $(u_i)$-trees and $T_0$ is fully  labelable, then $T_1$ is also
fully  labelable. These definitions can be easily adapted to finite 
$(u_i)$-trees (where the height is the length of its longest leaf).
 Figure \ref{LYPdCu3iua} shows some examples of fully labelable finite $(u_i)$-trees, while Figure
\ref{weDpXjQClZ} shows some examples of finite $(u_i)$-trees which are not fully labelable.

Clearly any $(u_i)$-tree which is isomorphic to the full binary tree, is 
fully labelable. The success of the  Ku\v{c}era coding was based on this fact, along with the fact
that a \pz class of sufficient measure contains such a canonical tree (subject to the growth of $(u_i)$). 
A similar remark can be made about the slightly more general G\'{a}cs coding.
We have already demonstrated that the density property that guarantees the extension property
cannot be expected to hold if the growth of $(u_i)$ is significantly less than $n+\sqrt{n}\cdot \log n$.
Hence more efficient coding methods, such as the one behind Lemma \ref{Dedk3zkhlL},
need to rely on a wider class of labelable trees. 

Given two trees $T_0, T_1$ (thought of as partially ordered sets), we say that $T_0$ is {\em splice-reducible} to $T_1$ if
we can obtain $T_1$ from $T_0$ via a series of operations on the nodes of $T_0$, each
consisting of splicing two sibling nodes
into one -- i.e.\ the two sibling nodes $u_0$ and $u_1$ are replaced by a single node $u$, for which the set of elements $>u$ is isomorphic to  the set of nodes strictly greater than $u_0$ union the set of nodes strictly greater than  $u_1$. 
The following result points to a concrete difference between Ku\v{c}era coding and the general
coding from  \cite{optcod}: in Ku\v{c}era coding the code-tree is an isomorphic copy of the full binary tree,
while in  \cite{optcod} the code-tree is only splice-reducible to an isomorphic copy of the 
full binary tree.\footnote{A similar remark can be made with respect to the G\'acs coding, only that instead of
binary trees we need to consider a homogeneous trees, in the sense that for each level, every node of that level
has the same number of successors.}

\begin{thm}
Given a $(u_i)$-tree $T$, 
the following are equivalent:
\begin{enumerate}[\hspace{0.5cm}(a)]
\item $T$ is a fully labelable $(u_i)$-tree;
\item $T$ is splice-reducible to an isomorphic copy of the full binary tree.
\end{enumerate}
\end{thm}
\begin{proof}
Suppose that $T$ is fully labelable.
We describe how to produce the full binary tree by a repeated application  
of the splice operation between siblings of $T$. Fix a full labelling of $T$ and 
obtain the minimal fully labeled tree $T'$ from $T$ by splicing the unlabelled nodes of $T$ onto
labelled ones. Now all nodes of $T'$ are labelled. Then gradually, starting from the first level
and moving toward the last level of $T'$, splice siblings with identical labels. Inductively,
by the properties of the assumed labelling, the resulting  $(u_i)$-tree is isomorphic to the
full binary tree.

Conversely, assume that $T$ is  splice-reducible to a  $(u_i)$-tree which is isomorphic to the full binary tree.
Then reversing the splice operations behind this reduction, we get a sequence of
node splitting operations that transform an isomorphic  $(u_i)$-copy of the full binary tree into $T$. Since this  $(u_i)$-copy of the full
binary tree has a full labeling, by making these labels persistent during
the series of splitting operations that lead to $T$, we get a full 
labelling of $T$.
\end{proof}

The work in \cite{optcod} shows that if $(u_i)$ is an increasing computable sequence,
then any tree of measure more than $\sum_i 2^{i-u_i}<\mu(\PP)$ has a full labelling.
Moreover, such a labelling has a \pz approximation, given any \pz approximation of $\PP$.
\begin{figure}
  \begin{center}
    \mbox{
      \subfigure{\scalebox{0.9}{
\begin{tikzpicture}[
    grow                    = up,
    level distance          = 0.5cm,
    level 1/.style={sibling distance=1.1cm},
   level 2/.style={sibling distance=0.5cm},
   level 3/.style={sibling distance=0.2cm},
   every node/.style={fill=red!60,circle,inner sep=1pt},
  ]
	\node  {}
		child {
			node{} 
				child {node {}  child {node {}}}
				child {node {} child {node {}} child {node {}} child {node {}}}
			} 
		child {
			node{} 
				child {node {} child {node {}} child {node {}}}
				child {node {} child {node {}} child {node {}}}
			} ;
\end{tikzpicture}
}} 
\hspace{1cm}
      \subfigure{\scalebox{0.9}{
\begin{tikzpicture}[
    grow                    = up,
    level distance          = 0.5cm,
    level 1/.style={sibling distance=1cm},
   level 2/.style={sibling distance=0.5cm},
   level 3/.style={sibling distance=0.2cm},
   every node/.style={fill=red!60,circle,inner sep=1pt},
  ]
	\node  {}
		child {
			node{} child {node {} child {node {}}
			}
			child {node {} child {node {}}
			child {node {}}}
			} 
		child {node{} child {
				node {} child {node {}} child {node {}} child {node {}}}
			child {node {} child {node {}}
			child {node {}}}
			} ;
\end{tikzpicture}
}}
\hspace{1cm}
      \subfigure{\scalebox{0.9}{
\begin{tikzpicture}[
    grow                    = up,
    level distance          = 0.5cm,
    level 1/.style={sibling distance=1.2cm},
   level 2/.style={sibling distance=0.4cm},
   level 3/.style={sibling distance=0.2cm},
   every node/.style={fill=red!60,circle,inner sep=1pt},
  ]
	\node  {}
		child {
			node{} 
				child {node {} 
						child {node {}}}
				child {node {} 
						child {node {}}}
				child {node {} 
						child {node {}}}
				child {node {} 
						child {node {}} 
						child {node {}}}
			} 
		child {
			node{} 
				child {node {} 
						child {node {}}}
				child {node {} 
						child {node {}} 
						child {node {}}}
			} ;
\end{tikzpicture}
}}
\hspace{1cm}
      \subfigure{\scalebox{0.9}{
\begin{tikzpicture}[
    grow                    = up,
    level distance          = 0.5cm,
    level 1/.style={sibling distance=1cm},
   level 2/.style={sibling distance=0.4cm},
   level 3/.style={sibling distance=0.2cm},
   every node/.style={fill=red!60,circle,inner sep=1pt},
  ]
	\node  {}
		child {
			node{} 
				child {node {} 
						child {node {}}}
				child {node {} 
						child {node {}} 
						child {node {}}}
				child {node {} 
						child {node {}} 
						child {node {}}}
			} 
		child {
			node{} 
				child {node {} 
						child {node {}}}
				child {node {} 
						child {node {}} 
						child {node {}}}
			} ;
\end{tikzpicture}
}}
      }
\end{center}
\caption{Some $(u_i)$-trees of length 3 which are not fully labelable.}
\label{weDpXjQClZ}
\end{figure}
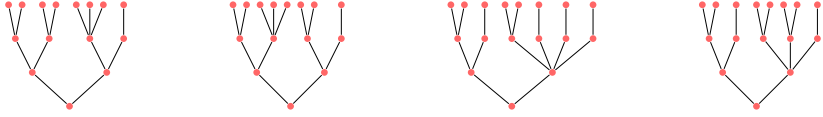

\newcommand{\etalchar}[1]{$^{#1}$}

\end{document}